\documentclass[11pt,leqno]{amsart}
\usepackage{amscd}
\usepackage{color}
\usepackage[symbol]{footmisc}
\usepackage{amssymb}
\usepackage{amsfonts}
\usepackage{latexsym}
\usepackage{verbatim}

\newcommand{\R}{\mathbb{R}}
\newcommand{\C}{\mathbb{C}}
\renewcommand{\H}{\mathbb{H}}
\newcommand{\Ric}{\mathrm{Ric}}

\renewcommand{\epsilon}{\varepsilon}
\renewcommand{\theta}{\vartheta}
\renewcommand{\phi}{\varphi}
\renewcommand{\Re}{\mathrm{Re}}
\renewcommand{\Im}{\mathrm{Im}}
\DeclareMathOperator*{\osc}{osc}

\theoremstyle{plain}
\newtheorem{teor}{Theorem}[section]
\newtheorem{prop}[teor]{Proposition}
\newtheorem{lem}[teor]{Lemma}

\newtheorem{conj}[teor]{Conjecture}
\newtheorem{oss}[teor]{Remark}

\theoremstyle{definition}

\title{The parabolic quaternionic Calabi-Yau equation on hyperk\"ahler manifolds}
\makeatletter
\@namedef{subjclassname@2020}{%
 \textup{2020} Mathematics Subject Classification}
\makeatother
\begin{document}
\thanks{This work was supported by GNSAGA of INdAM}
\subjclass[2020]{53C26, 35K96, 53E30}

\address{(Lucio Bedulli) Dipartimento di Ingegneria e Scienze dell'Informazione e Matematica, Universit\`a dell'Aquila,
via Vetoio, 67100 L'Aquila, Italy}
\email{lucio.bedulli@univaq.it}

\address{(Giovanni Gentili) Dipartimento di Matematica e Informatica ``Ulisse Dini'', Universit\`{a} degli Studi di Firenze,
viale Morgagni 67/a, 50134 Firenze, Italy}
\email{giovanni.gentili@unifi.it}

\address{(Luigi Vezzoni) Dipartimento di Matematica G. Peano \\ Universit\`a di Torino\\
Via Carlo Alberto 10\\
10123 Torino\\ Italy}
\email{luigi.vezzoni@unito.it}

\author{Lucio Bedulli, Giovanni Gentili and Luigi Vezzoni}

\date{\today}

\maketitle

\begin{abstract}
We show that the parabolic quaternionic Monge-Amp\`ere equation on a compact hyperk\"ahler manifold has always a long-time solution which once normalized converges smoothly to a solution of the quaternionic Monge-Amp\`ere equation. This is the same setting in which Dinew and Sroka \cite{Dinew-Sroka} prove the conjecture of Alesker and Verbitsky \cite{Alesker-Verbitsky (2010)}. We also introduce an analogue of the Chern-Ricci flow in hyperhermitian manifolds. 
\end{abstract}

\section{Introduction}
A hypercomplex manifold is a real $4n$-dimensional smooth manifold equipped with three complex structures $I,J,K$ satisfying the quaternionic-type relations $IJ=-JI=K$. A hyperhermitian metric $g$ is determined either by a form $\Omega$ of type $(2,0)$ with respect to $I$ or by a positive real form $\omega$ of type $(1,1)$ with respect to $I$ such that $J\omega=-\omega$. A hyperhermitian metric is called HKT (hyperk\"ahler with torsion) if $\partial \Omega=0$ (the operators $\partial$ and $\bar\partial$ will be always taken with respect to the complex structure $I$ throughout the paper). HKT metrics were first introduced in \cite{Howe-Papadopoulos (2000)} with motivations coming from theoretical physics and further studied from the purely geometric point of view (see e.g. \cite{Alesker (2013),Alesker-Shelukhin (2017),Alesker-Verbitsky (2010),BDV,BGV,Dinew-Sroka,Dotti-Fino (2000),Dotti-Fino (2002),GV,GLV,Grantcharov-Poon (2000),Sroka,Sroka22,Verbitsky (2002),Verbitsky (2009),Zhang} and the references therein). 

In hyperhermitian geometry the role of plurisubharmonic functions is usually replaced by smooth real functions $\varphi$ satisfying $\Omega+\partial \partial_J\varphi>0$ ({\em quaternionic plurisubharmonic} functions), where the positivity is in the sense of (2,0)-forms (see the preliminaries) and $\partial_J=J^{-1}\bar \partial J$ is the {\em twisted $\bar\partial$ operator}. Equivalently a function $\varphi$ is quaternionic plurisubharmonic if $\omega + \tfrac{i}{2}( \partial \bar \partial \phi - J\partial \bar \partial \phi)$ is positive as a $(1,1)$-form. 

In \cite{Alesker-Verbitsky (2010)} Alesker and Verbitsky introduced an analogue of the complex Monge-Amp\`ere equation on hyperhermitian manifolds by considering for a given smooth function $f$  the {\em quaternionic  Monge-Amp\`ere equation} 
\begin{equation}\label{QMA}
(\Omega+\partial\partial_J\varphi)^n={\rm e}^{f+b}\,\Omega^n\,,\qquad \Omega+\partial\partial_J\varphi>0\,, \qquad  \sup_M \phi =0\,,
\end{equation}
for a real-valued function $\varphi$ and a constant $b$. The equation can be reformulated in terms of real $2$-forms as 
\begin{equation}\label{QMA2}
(\omega + \tfrac{i}{2}( \partial \bar \partial \phi - J\partial \bar \partial \phi))^{2n}={\rm e}^{2(f+b)}\,\omega^{2n}\,,\qquad \omega + \tfrac{i}{2}( \partial \bar \partial \phi -J\partial \bar \partial \phi)>0\,, \qquad  \sup_M \phi =0\,.
\end{equation}
In \cite{Alesker-Verbitsky (2010)} it is conjectured that the equation is always solvable on compact HKT manifolds at least under the extra assumption that the canonical bundle of $(M,I)$ is holomorphically trivial. 

\medskip
Following the strategy of Yau for proving the Calabi conjecture \cite{Yau}, the natural approach for studying the quaternionic Monge-Amp\`ere equation is via the continuity method and the hard part in this direction is, as usual, the proof of a priori estimates. Fortunately some important results have been established in this direction. The $C^0$ estimate is now proved in the most general case. A first proof of the $C^0$ estimate was given in \cite{Alesker-Verbitsky (2010)} under the assumptions on $g$ to be HKT and on the canonical bundle of $(M,I)$ to be holomorphically trivial. Furthermore the $C^0$ estimate was improved in \cite{Alesker-Shelukhin (2017)} by removing the condition on the canonical bundle and  recently also the HKT assumption was removed by Sroka in \cite{Sroka22} applying a recent method of Guo, Phong  and Tong \cite{GP22a,GP22b,GPT21}. The higher order estimates have been established so far only under extra assumptions. In \cite{Alesker (2013)} Alesker confirmed the conjecture on compact flat hyperk\"ahler manifolds. More recently Dinew and Sroka drastically improved Alesker's result by proving the conjecture also for non-flat hyperk\"ahler manifolds \cite{Dinew-Sroka}.  Other partial confirmations to the conjecture can be found in \cite{GV,GV2}. 

%

\medskip 
In the present paper we first observe that in the most general case an upper bound of the Laplacian of the solution to \eqref{QMA} combined with the $C^0$ estimate implies all the other a priori estimates (Theorem \ref{C2alphaelliptic}). In particular this allows us to simplify part of the proof of the Dinew-Sroka Theorem in \cite{Dinew-Sroka} (see Remark \ref{semplification}). We further prove an analogous result for the quaternionic parabolic Monge-Amp\`ere equation   
\begin{equation}\label{pMA}
\dot \phi= 2\log \frac{(\Omega+\partial\partial_J\varphi)^n}{\Omega^n}-2f\,,\qquad \Omega+\partial\partial_J\varphi>0\,, \qquad \phi(\cdot,0)=0\,,
\end{equation}
introduced in \cite{BGV,Zhang} (see Theorem \ref{C2alpha}). This leads to the following theorem which we prove in section \ref{secproof}

\begin{teor}\label{main}
Let $(M,I,J,K,g)$ be a compact hyperhermitian manifold. Assume that $(I,J,K)$ admits a compatible hyperk\"ahler metric $\hat g$. Then for every $f\in C^{\infty}(M)$ equation \eqref{pMA} has a unique long-time solution $\varphi(t)$, $t\in [0,\infty)$. The normalization
$$
\tilde \varphi:=\varphi-\frac{1}{\int_M \Omega^n \wedge \bar\Omega^n}\int_M\varphi \,\Omega^n \wedge \bar\Omega^n
$$
converges to a solution of \eqref{QMA} with a suitable choice of $b$.
\end{teor} 

Theorem \ref{main} is the natural generalization of the main theorem in \cite{BGV,Zhang} and the parabolic version of the main result in \cite{Dinew-Sroka}. 

 Given a solution $\varphi(t)$ to \eqref{pMA}, the associated $ (1,1) $-form $\omega(t):=\omega+ \tfrac{i}{2}( \partial \bar \partial \phi(t) - J\partial \bar \partial \phi(t))$ satisfies the evolution equation
\begin{equation}\label{JRicciflow}
\dot\omega(t)=-\frac{1}{2}\left({\rm Ric}(\omega(t))-J {\rm Ric}(\omega(t))-\beta+J\beta\right)
\end{equation}
where $\beta={\rm Ric}(\omega)-2i\partial \bar \partial f$ and ${\rm Ric}$ is the Chern-Ricci form with respect to $I$.  Long-time existence and convergence of \eqref{pMA} would imply the fact that for every representative $ \beta $ of $ c_1^{\rm BC}(M,I)$, there is a unique hyperhermitian metric of the form $ \omega_\phi=\omega+ \tfrac{i}{2}( \partial \bar \partial \phi - J\partial \bar \partial \phi) $ with ${\rm Ric}(\omega_\phi)=\beta$. From this point of view Theorem \ref{main} is in the spirit of Cao's theorem \cite{Cao}.

Equation \eqref{JRicciflow} suggests to consider the geometric flow 
\begin{equation}
\label{J-}
\dot \omega(t)=-\frac{1}{2}\left({\rm Ric}(\omega(t))-J {\rm Ric}(\omega(t))\right)\,,\quad \omega(0)=\omega\,,
\end{equation}
since it preserves the compatibility with the hypercomplex structure and plays the role of the Ricci-flow in K\"ahler geometry and of the Chern-Ricci flow in Hermitian geometry \cite{TW}. We think that a study of flow \eqref{J-} in the same spirit of  \cite{TW} could give new insights in hyperhermitian geometry. In section \ref{sJ-} we collect some preliminary observations on the flow. 


\bigskip  
\noindent {\bf Acknowledgements.} We are very grateful to M. Sroka for several fruitful discussions on earlier versions of this paper. Moreover we are grateful to D. H. Phong and M. Garcia-Fernandez for the interest shown in our research. We also thank E. Fusi for interesting conversations and remarks.    

\section{Preliminaries}\label{Preliminaries}
A {\em hypercomplex manifold} $(M,I,J,K)$ is an even-dimensional complex manifold $(M,I)$ equipped with two additional complex structures $(J,K)$ satisfying the quaternionic-type relations $IJ=-JI=K$. A Hermitian metric $g$ on $(M,I,J,K)$ is \emph{hyperhermitian} if it is compatible with each $I,J,K$. 
Any hyperhermitian metric $ g $ induces a corresponding form
\[
\Omega:=\frac12 \left(g(J\cdot,\cdot)+ig(K\cdot,\cdot)\right) = \frac12 \left(\omega_J + i \omega_K\right)\in \Lambda^{2,0}_I
\]
which satisfies the following properties 
\begin{itemize}
\item $ \Omega(J\cdot,J\cdot)=\bar \Omega $ ($\Omega$ is q-real);

\vspace{0.1cm}
\item $ \Omega(Z,J\bar Z)>0 $ for every non-zero $ Z\in T^{1,0}M $ ($\Omega$ is  positive).
\end{itemize}
Conversely, any q-real and positive $ \Omega\in \Lambda_I^{2,0}$ induces a hyperhermitian metric $ g $ via the relation 
\begin{equation}\label{gfromOmega}
g=2\Re (\Omega(\cdot,J\cdot))\,.
\end{equation}
Hence we have a one-to-one correspondence between  $q$-real positive $(2,0)$-forms and hyperhermitian metrics. If $ \Omega $ is $ \partial $-closed we say that  $g$ is \emph{hyperk\"ahler with torsion} (HKT). We further denote by $\omega$ the fundamental form of $(g,I)$. We have the following relation 
\begin{equation}\label{frac}
\frac{\Omega^n\wedge \bar \Omega^n}{(n!)^2}=\frac{\omega^{2n}}{(2n)!}\,,
\end{equation}
(see e.g. \cite[Section 4.3]{SThesis}).

%
%

%


Let $\partial$ be the $\partial$-operator with respect to $I$ and $\partial_J:=J^{-1}\bar\partial J\colon \Lambda^{r,0}_I\to \Lambda^{r+1,0}_I$.  Then 
$$
\partial\partial_J=-\partial_J\partial\,,
$$
see \cite{Dinew-Sroka}.
Every $\Omega_\phi:= \Omega + \partial \partial_J \phi >0$ induces a hyperhermitian metric $g_\varphi$. We further denote by $\omega_\varphi$ the $ (1,1) $-form of $(g_\varphi,I)$.  
 
 \medskip 
The following useful lemma follows from \cite[Remark 4.1]{Sroka22} and \cite[Proposition 2.15]{Dinew-Sroka} but we prove for the reader's convenience

\begin{lem}\label{dedeJ}
For every $\phi\in C^{\infty}(M)$ we have 
\[
\partial\partial_J\phi(X,Y)=-\frac{1}{2}\left( \partial \bar \partial \phi(X,JY)+ \partial \bar \partial \phi(JX,Y) \right)^{2,0}\,. 
\]
Moreover, 
\[
\omega_\phi=\omega + \frac{i \partial \bar \partial \phi - iJ\partial \bar \partial \phi}{2}
\]
and 
$$
\mathrm{tr}_{g_\phi}g=2n-\Delta_{\phi} \phi\,,
$$
where $\Delta_\phi$ is the Chern-Laplacian operator with respect to $ g_\phi $. 
\end{lem}
\begin{proof}
The first part of the statement can be easily proved locally by using $ I $-holomorphic coordinates such that $ \partial_i J_s^{\bar r}=\partial_s J_i^{\bar r} $ \cite[Remark 2.13]{Dinew-Sroka}. Indeed, 
\begin{equation*}\label{locddJ}
\begin{aligned}
\partial\partial_J\phi&=-\partial J\bar \partial \phi=
-\partial J \phi_{\bar r} d\bar z^r=-\partial  (\phi_{\bar r} J^{\bar r}_{s} dz^{s})
\\
&=- \phi_{i\bar r}J^{\bar r}_{s} dz^{i}\wedge dz^s- 
\phi_{\bar r}\partial_iJ^{\bar r}_{s} dz^{i}\wedge dz^s
=- \phi_{i\bar r}J^{\bar r}_{s} dz^{i}\wedge dz^s\,;
\end{aligned}
\end{equation*}
hence 
\begin{equation}\label{ddJ}
\partial\partial_J\phi(X,Y)=\frac{1}{2}\left( -\partial \bar \partial \phi(X,JY)+ \partial \bar \partial \phi(Y,JX) \right)^{2,0}
\end{equation}
as required. Moreover, from \eqref{gfromOmega} and \eqref{ddJ} we have
\[
\begin{aligned}
g_\phi(X,Y)=2\Re(\Omega_\phi(X,JY))
=g(X,Y)-\frac{1}{2}\left(i\partial \bar \partial \phi(KX,JY)-i\partial \bar \partial \phi(X,IY)\right)
\end{aligned}
\]
and thus
\[
\begin{aligned}
\omega_\phi(X,Y)=g_\phi(IX,Y)
=\omega(X,Y)-\frac{1}{2}\left(i\partial \bar \partial \phi(JX,JY)-i\partial \bar \partial \phi(X,Y)\right)\,.
\end{aligned}
\]
Finally, since $ g_\phi $ is $ J $-Hermitian we have
\[
\mathrm{tr}_{g_\phi}g=\mathrm{tr}_{\omega_\phi} \omega=2n- \mathrm{tr}_{\omega_\phi} \frac{i \partial \bar \partial \phi - iJ\partial \bar \partial \phi}{2}=2n-\Delta_{\phi} \phi
\]
as claimed.
\end{proof}
 
The following Lemma essentially follows from \cite{Verbitsky (2009)} 

\begin{lem} \label{lemma3}
If $g$ is HKT, then 
$$
d\Omega^n=\theta\wedge \Omega^{n}
$$
where $\theta=-Id^*\omega$ is the Lee form of $(\omega,I)$. In particular ${\rm Ric}(\omega)=dd^*\omega $ and in the compact case
$$
d \Omega^n=0\iff \omega \mbox{ is balanced } \iff {\rm Ric}(\omega)=0\,.
$$
\end{lem}
 \begin{proof}
It is quite easy to observe that 
\[
*\Omega=\frac{1}{(n-1)!\,n!}\Omega^{n}\wedge \bar \Omega^{n-1}
\]
and 
$$
*\alpha=-\frac{1}{(n-1)!\,n!}J\alpha \wedge \Omega^{n}\wedge \bar \Omega^{n-1}
$$
for every $1$-form $\alpha$ of type $(1,0)$ with respect to $I$. The HKT condition implies $ \theta=-Jd^*\omega_J=-Kd^*\omega_K $ \cite{IPap}, hence
\[
\begin{split}
J\theta&=d^*\omega_J=d^*(\Omega+\bar \Omega)=-*d*(\Omega+\bar \Omega)=-\frac{1}{(n-1)!\,n!}*d\left(\Omega^{n}\wedge \bar \Omega^{n-1}+\Omega^{n-1}\wedge \bar \Omega^n\right)\\
&=-\frac{1}{(n-1)!\,n!}*\left(d \Omega^{n}\wedge \bar \Omega^{n-1}+\Omega^{n-1}\wedge d\bar \Omega^n\right)\,.
\end{split}
\]
Clearly there exists a $ (0,1) $-form $ \beta $ such that $ d\Omega^n=\beta \wedge \Omega^n $, but from these computations it follows that
\[
J\theta=-\frac{1}{(n-1)!\,n!}*\left(\beta \wedge  \Omega^{n}\wedge \bar \Omega^{n-1}+\bar \beta \wedge \Omega^{n-1}\wedge \bar \Omega^n\right)=J(\beta +\bar \beta)
\]
and thus $\beta \wedge \Omega^n=\theta \wedge \Omega^n$ as claimed. Moreover since $\omega$ is HKT it is Bismut-Ricci flat and formula (2.7) in \cite{AlexIvanov} together implies ${\rm Ric}(\omega)=dd^*\omega $. The last statement is trivial.
\end{proof}

\begin{oss}\label{rem}{\rm 
In \cite{BDV} it is proved that on a compact nilmanifold $(N/\Gamma,I,J,K)$ a left-invariant HKT metric is always balanced. This fact can be also deduced from Lemma \ref{lemma3} taking into account that the Chern-Ricci form of a left-invariant Hermitian metric on a complex nilmanifold is always zero (see e.g. \cite[Proposition 2.1]{LV}).} 
\end{oss}

\section{From a bound of the Laplacian to $C^{2,\alpha}$-estimates}
In \cite{TWWY} it is proved a general theorem for deducing $C^{2,\alpha}$  estimates of a solution of an elliptic equation from a bound on the Laplacian of the solution. The theorem is applied in \cite{TWWY} to a large class of 
equations in Hermitian geometry. In this section we observe that it can be also applied to the quaternionic Monge-Amp\`ere equation.

\begin{teor}\label{C2alphaelliptic}
Let $(M,I,J,K,g)$ ba a compact hyperhermitian manifold and $\varphi$ a solution to \eqref{QMA} such that
\begin{equation}\label{Laplell}
\|\varphi\|_{C^0}\leq C\,\, \mbox{ and }\,\,\mathrm{tr}_g g_\phi \leq C\,,
\end{equation}  
for some constant $C>0$. Then for $0<\alpha<1$ there exists a constant $C_\alpha>0$ depending only on $ (M,I,J,K,g) $, $f$, $\alpha$ and $ C $ such that 
$$
\|\nabla^2\varphi\|_{C^{\alpha}}\leq C_\alpha\,.
$$  
\end{teor}
\begin{proof}
Choose $ I $-holomorphic local coordinates $ (z^1,\dots,z^{2n}) $ in a chart, which for simplicity we identify with the unit ball $ B_1\subseteq \C^{2n} $. Consider also the underlying real coordinates $ (x^1,\dots,x^{4n}) $ given by $ z^k=x^k+ix^{2n+k} $ for $ k=1,\dots,2n $ and the usual real representation of complex matrices defined as
\[
\iota(H):=\begin{pmatrix}
\,\,\,\,\Re(H) & \Im(H)\\
-\Im(H) & \Re(H)
\end{pmatrix}\,.
\]
Let $ \mathrm{Herm}(2n) $ and $ \mathrm{Sym}(4n) $ be the spaces of $ 2n\times 2n $ Hermitian matrices and $ 4n\times 4n $ real symmetric matrices respectively. Notice that $ \iota $ sends $ \mathrm{Herm}(2n) $ to $ \mathrm{Sym}(4n) $. We define the following functions
\begin{itemize}
\itemsep0.3em
\item $ \mathcal{F}\colon \mathrm{Sym}(4n) \to \R $ given by $ \mathcal{F}(N):=\frac{1}{2}\log \det(N) $;
\item $ \mathcal{S} \colon B_1 \to \mathrm{Sym}(4n) $ given by $ \mathcal{S}(x):=\iota(g(x)) $;
\item $ \mathcal{T}\colon \mathrm{Sym}(4n) \times B_1 \to \mathrm{Sym}(4n) $ given by $ \mathcal{T}(N,x):=\frac{1}{4}(p(N)+\iota(\,^t\!J(x))p(N)\iota(J(x))) $, where $ p $ is the projection $ p(N):=\frac{1}{2} (N+\,^t\!I NI) $.
\end{itemize}
Here, we are writing $ g(x) $ and $ J(x) $ for the complex matrices of $ g $ and $ J $ at the point $ x $ in the coordinates $ (z^1,\dots,z^{2n}) $. For simplicity we set $ \tilde J=\iota(J) $. Since $ p(D^2_\R u(x))=2\iota(D^2_\C u(x))$ (here $ D^2_\R $ and $ D^2_\C $ are the real and complex Hessian, respectively) for any function $ u\colon B_1 \to \R $ and $ \det(\iota(H))=\det(H)^2 $ for any Hermitian matrix $ H $, we have
\[
\begin{aligned}
\mathcal{F}(\mathcal{S}(x)+\mathcal{T}(D^2_\R\phi(x)))&=\frac{1}{2}\log \det \left( \iota (g_{i\bar j})+\frac{1}{2}\iota (\phi_{i\bar j})+\frac{1}{2}\tilde J_i^{\bar s}\iota (D^2_\C\phi)_{r\bar s} \tilde J_{\bar j}^r\right)(x)\\
&=\frac{1}{2}\log \det \iota \left(g_{i\bar j}+\frac{1}{2}\phi_{i\bar j}+\frac{1}{2}J_i^{\bar s}(D^2_\C\phi)_{r\bar s} J_{\bar j}^r\right)(x)\\
&=\log \det \left(g_{i\bar j}+\frac{1}{2}\phi_{i\bar j}+\frac{1}{2}J_i^{\bar s}(D^2_\C\phi)_{r\bar s} J_{\bar j}^r\right)(x)\\
&=2f(x)+2b+\log \det g(x)\,.
\end{aligned}
\]

The arithmetic-geometric means inequality gives $ \mathrm{tr}_g{g_\phi}\geq 2n (\frac{\det g_\phi}{\det g})^{1/2n} =2n\mathrm{e}^{(f+b)/n}\geq C' $, because $|b|$ is bounded by $\sup |f|$ using a standard maximum principle argument directly on equation \eqref{QMA}. Since also $ \mathrm{tr}_{g}g_\phi\leq C $ by assumption \eqref{Laplell}, we then have
\[
C_0^{-1}(\delta_{i\bar j})\leq g_{i\bar j}(x)+\frac{1}{2}\phi_{i\bar j}(x)+\frac{1}{2}J_i^{\bar s}(x)\phi_{r\bar s}(x) J_{\bar j}^r(x)\leq C_0(\delta_{i\bar j})
\]
for $ x\in B_1 $ and a constant $ C_0>0 $. Since $ \iota $ preserves (semi)positivity, i.e. $ H_1\leq H_2 $ if and only if $ \iota(H_1)\leq \iota(H_2) $, we deduce
\[
C_0^{-1}(\delta_{i j})\leq \mathcal{S}(x)+\mathcal{T}(D^2_\R\phi(x))\leq C_0(\delta_{i j})\,.
\]
Let $ \mathcal{E} $ denote the compact convex subset
\[
\mathcal{E}:=\{ N\in \mathrm{Sym}(4n) \mid C_0^{-1}(\delta_{ij})\leq N\leq C_0(\delta_{ij}) \}\,.
\]
We check that all the assumptions $ \mathbf{H1}-\mathbf{H3}$ of \cite[Theorem 1.2]{TWWY} are satisfied.
\begin{itemize}
\itemsep0.3em
\item It is well-known that $ \mathcal{F} $ is uniformly elliptic and concave on $ \mathcal{E} $ (conditions $ \mathbf{H1}.(1) $ and $ \mathbf{H1}.(2) $ of \cite[Theorem 1.2]{TWWY}). Moreover $ \mathbf{H1}.(3) $ is trivial for $ \mathcal{F} $ since it does not depend on $ x $.
\item Next we verify conditions $ \mathbf{H2} $ in \cite[Theorem 1.2]{TWWY}. Condition $ \mathbf{H2}.(1)$ is easily checked and $ \mathbf{H2}.(2) $ is straightforward. We just need to show that also $ \mathbf{H2}.(3)$ holds. For any positive semidefinite $N \in \mathrm{Sym}(4n) $ and $v\in \R^{4n}$ we have 
\[
\begin{split}
\quad \qquad \langle \mathcal{T}(N,x)v, v\rangle&=\frac{1}{8}\left( \langle Nv, v\rangle +\langle \,^t\!INIv,v\rangle+\langle  \,^t\!\tilde
JN\tilde
Jv,v\rangle+\langle \,^t\!\tilde
J\,^t\!INI\tilde
 Jv,v\rangle \right)\!(x)\\
&=\frac{1}{8}\left( \langle Nv, v\rangle +\langle NIv,Iv\rangle+\langle N\tilde Jv,\tilde Jv\rangle+\langle NI\tilde Jv, I \tilde Jv\rangle \right)\!(x)\geq 0\,.
\end{split}
\]
We may assume without loss of generality  that  $ J(0) $ is orthogonal and we get  
\[
\frac{1}{8}\|N\|\leq \|\mathcal{T}(N,0)\|\leq \frac12 \|N\|\,,
\]
where $ \|A\|=\sup_{\|v\|=1} \langle Av,v\rangle $. Possibly shrinking the ball, we may assume that $ J(x) $ is close to $ J(0) $  and $ \mathbf{H2}.(3) $ is satisfied.

\item Condition $ \mathbf{H3} $ obviously holds. 
\end{itemize}  
Since the assumptions $ \mathbf{H1}-\mathbf{H3}$ of \cite[Theorem 1.2]{TWWY} are verified the theorem follows.
\end{proof}

\begin{oss}\label{semplification}{\rm 
Theorem \ref{C2alphaelliptic} was already proved by Alesker in the case of compact locally flat HKT manifolds \cite{Alesker (2013)}. Our version allows to simplify the proof of the main theorem of \cite{Dinew-Sroka}. Indeed the proof of the Alesker-Verbitsky conjecture on hyperk\"ahler manifolds is obtained in \cite{Dinew-Sroka}  proving independently the $C^1$ estimate and a bound for the Laplacian  and then  combining them in order to obtain the second order estimate. Hence the proof of the Dinew and Sroka theorem can be alternatively obtained bypassing the gradient estimate and using our Theorem \ref{C2alphaelliptic}. We also note that Theorem \ref{C2alphaelliptic} does not need $g$ to be HKT.}
\end{oss}

Next we focus on the \lq\lq parabolic counterpart\rq\rq of Theorem \ref{C2alphaelliptic}:
\begin{teor}\label{C2alpha}
Let $0<\alpha<1$, $\varphi(t)$, $t\in[0,T)$, a solution  to \eqref{pMA} such that
\begin{equation}\label{Lapl}
\osc_M \varphi \leq C\,, \quad \|\dot \varphi\|_{C^0}\leq C\,, \quad
\mathrm{tr}_g g_\phi \leq C\,,
\end{equation}
for some positive constant $C$.
Let $\epsilon\in (0,T)$. 
Then $\varphi(t)$ satisfies the following a priori estimate
$$
\|\nabla^2\varphi\|_{C^{\alpha}}\leq C_\alpha
$$  
in $[\epsilon,T)$, where $C_\alpha>0$ depends only on $ (M,I,J,K,g) $, $f$, $\alpha$, $ C $ and $\epsilon$.
\end{teor}
\begin{proof}
Here we apply the general result of Chu \cite[Theorem 5.1]{Chu}, in the same fashion as \cite[Lemma 6.1]{Chu2}. In the same notations of the previous theorem, with $ \mathcal{F},\mathcal{S} $ and $ \mathcal{T} $ chosen in the same way, we have
\[
\begin{aligned}
\dot \phi(x,t) -\mathcal{F}(\mathcal{S}(x)+\mathcal{T}(D^2_\R\phi(x,t)))
=-2f(x)-\log \det g(x)\,.
\end{aligned}
\]

From \eqref{Lapl}, the arithmetic-geometric means inequality and Lemma \ref{C0} we get $ C\geq \mathrm{tr}_g{g_\phi}\geq 2n (\frac{\det g_\phi}{\det g})^{1/2n} =2n\mathrm{e}^{(\dot \phi +2f)/2n}\geq C^{-1} $. We then have
\[
C_0^{-1}(\delta_{i\bar j})\leq g_{i\bar j}(x)+\frac{1}{2}\phi_{i\bar j}(x,t)+\frac{1}{2}J_i^{\bar s}(x)\phi_{r\bar s}(x,t) J_{\bar j}^r(x)\leq C_0(\delta_{i\bar j})
\]
for $ (x,t)\in B_1\times (0,1] $ and a uniform constant $ C_0>0 $. We then infer
\[
C_0^{-1}(\delta_{i j})\leq \mathcal{S}(x)+\mathcal{T}(D^2_\R\phi(x,t))\leq C_0(\delta_{i j})\,.
\]
Let $ \mathcal{E} $ denote the compact convex subset
\[
\mathcal{E}:=\{ N\in \mathrm{Sym}(4n) \mid C_0^{-1}(\delta_{ij})\leq N\leq C_0(\delta_{ij}) \}\,.
\]
All the assumptions $ \mathbf{H1}-\mathbf{H3}$ of \cite[Theorem 5.1]{Chu} are easily checked as in the previous theorem. Nonetheless, at this point, we cannot directly apply \cite[Theorem 5.1]{Chu} since $ \phi $ does not necessarily satisfy a $C^0$ a priori bound.  However, we can overcome this issue working  as in \cite[Lemma 6.1]{Chu2}. Here is where the constant $\epsilon$ plays a role. We consider the two cases $T<1$ and $T\geq 1$, separately. 

If $T<1$, we have a uniform $C^0$ bound for $\phi$ since 
\[
|\phi|=\left \lvert \int_0^t \dot \phi\, dt \right \rvert \leq T\sup_{M\times [0,T)} |\dot \phi|\leq C\,.
\]
In this case we can directly apply \cite[Theorem 5.1]{Chu} to conclude.
 
If  $T\geq 1$ we consider, for any fixed $a\in (0,T-1)$, the following auxiliary function
\[
\phi_a(x,t):=\phi(x,t+a)-\inf_{B_1\times [a,a+1)}\phi\,, \qquad t\in [0,1)\,.
\]
Clearly we have $\|\phi_a\|_{C^0}\leq \osc_M \phi \leq C$. Moreover, from \eqref{pMA} we see that $ \phi_a $ satisfies the parabolic Monge-Amp\`{e}re equation
\[
\dot \phi_a=2\log \frac{(\Omega+\partial \partial_J \phi_a)^n }{\Omega^n}-2f\,.
\]
Since, from \eqref{Lapl} we know that $ \mathrm{tr}_{g}g_{\phi_a} $ is uniformly bounded from above we may apply \cite[Theorem 5.1]{Chu} to $\phi_a$ and deduce that for any fixed $\epsilon \in (0,\frac{1}{2})$ we have
\[
\|\nabla^2 \phi \|_{C^{\alpha}(B\times [a+\epsilon,a+1))} \leq \|\nabla^2 \phi_a \|_{C^{\alpha}(B_1\times [\epsilon,1))}\leq C\,,
\]
where $C$ is a uniform constant that depends on $\epsilon$ and $\alpha$. Since $a\in (0,T-1)$ is arbitrary we obtain the estimate
\[
\|\nabla^2 \phi \|_{C^{\alpha}(B_1\times [\epsilon,T))}\leq C\,,
\]
allowing us to conclude. 
\end{proof}

\begin{oss}\label{Schauder}
{\rm 
As usual in the elliptic case one can deduce higher order estimates from Theorem \ref{C2alphaelliptic}  by using a standard bootstrapping argument via Schauder estimates and obtaining that under the assumptions of Theorem \ref{C2alphaelliptic}  
$$
\|\nabla^{k}\varphi\|_{C^\alpha}\leq C_{k+\alpha}
$$
for constants $C_{k+\alpha}$ depending only on $ (M,I,J,K,g) $, $f$, $\alpha$, $k$ and $C$.

Analogously in the parabolic case under the assumptions of Theorem \ref{C2alpha}, $\varphi(t)$ satisfies 
$$
\|\nabla^k\varphi\|_{C^{\alpha}}\leq C_{k+\alpha} 
$$  
in $[\epsilon,T)$, where $C_{k+\alpha}$ depends only on $ (M,I,J,K,g) $, $f$,  $\alpha$, k, $\epsilon$ and $C$. 
}
\end{oss}

\section{Proof of Theorem \ref{main}}\label{secproof}
Let $(M,I,J,K,g)$ be a compact hyperhermitian manifold and, for $f\in C^{\infty}(M)$, consider the parabolic quaternionic Monge-Amp\`{e}re equation \eqref{pMA}.


For $ \phi \in C^\infty(M) $ let
\[
{\rm P}(\phi)=2\log \frac{\Omega_\phi^n}{\Omega^n}-2f\,.
\]
The first variation of ${\rm P}$ at $\varphi$ is 
\[
{\rm P}_{*\vert \phi}(\psi)=2n\frac{\partial \partial_J\psi\wedge \Omega_\varphi^{n-1}}{\Omega_\varphi^{n}}=\Delta_{\phi} \psi,
\]
where the last equality can be easily checked by using for instance $I$-holomorphic coordinates \cite[Remark 2.13]{Dinew-Sroka}. It follows that equation \eqref{pMA} is strictly parabolic. In particular it always admits a solution $\varphi(t)$, $t\in [0,\epsilon)$, for some $\epsilon$ small enough. 

\begin{lem}\label{equiv}
The quaternionic parabolic Monge-Amp\`ere equation \eqref{pMA} can be alternatively rewritten as
$$
\dot \varphi=\log \frac{\omega_{\varphi}^{2n}}{\omega^{2n}}-2f\,,\quad \varphi(0)=0\,. 
$$
\end{lem}
\begin{proof}
Since the ratio $\Omega_\varphi^n/\Omega^n$ is real, we have 

$$
\left(\frac{\Omega_\phi^n}{\Omega^n}\right)^2=\frac{\Omega_\phi^n\wedge\bar \Omega_\varphi^n}{\Omega^n\wedge\bar \Omega^n}=\frac{\omega_\varphi^{2n}}{\omega^{2n}}
$$
where in the last equality we used \eqref{frac}. The claim follows. 
\end{proof}

For a solution $ \phi(t) $ to \eqref{pMA} we shall also consider its normalization
$$
\tilde \phi(t):=\varphi(t)-\frac{1}{\int_M\Omega^n\wedge \bar \Omega^n}\int_M \varphi(t)\,\Omega^n\wedge \bar \Omega^n\,.
$$

\begin{lem}\label{C0}
A solution $\varphi(t)$ to \eqref{pMA} satisfies the following a priori estimates
\[
\|\dot \phi \|_{C^0}\leq C\,, \qquad \osc_M \phi \leq C\,, \qquad \|\dot{\tilde \phi} \|_{C^0}\leq C\,, \qquad \| \tilde  \phi\|_{C^0} \leq C\,, 
\]
for a uniform constant $C>0$ depending only on $ (M,I,J,K,g) $ and $ \|f\|_{C^0} $. 
\end{lem}
\begin{proof}
The technique of the proof is standard and for example analogous to that of \cite[Theorems 2.1 and 2.2]{Gill}; it relies on the elliptic $ C^0 $-estimate obtained by Sroka \cite{Sroka22} on compact hyperhermitian manifolds. We give some details for completeness. Differentiating \eqref{pMA} we get that $ \dot \phi $ satisfies
\[
\partial_t \dot \phi=   \Delta_\phi \dot \phi
\]
and from the parabolic maximum principle we infer that $ \|\dot \phi \|_{C^0}\leq C $ for some uniform positive constant $ C $. On the other hand, viewing $ \phi $ as a solution of the quaternionic Monge-Amp\`{e}re equation
\[
(\Omega+\partial \partial_J \phi)^n=\mathrm{e}^{\frac{\dot \phi }{2}+ f}\Omega^n
\]
with datum $ \frac{\dot \phi}{2} + f $ we may use the main theorem in \cite{Sroka22} to get $ \osc_M \phi \leq C $. The $ C^0 $ bound on the time derivative of $ \tilde \phi $ is then straightforward. Moreover by definition of $ \tilde \phi $ for every $t\in\mathbb R_+$ such that $M\times \{t\}$ is in the domain of $\varphi$, there exists $x_0\in M$ such  that $ \tilde \phi(x_0,t)=0 $. Therefore
\[
| \tilde \phi(x,t) |=| \tilde \phi(x,t)-\tilde \phi(x_0,t) |=|\phi(x,t)- \phi(x_0,t)| \leq C\,,
\]
and the $ C^0 $ bound on $ \tilde \phi $ follows.
\end{proof}

\begin{lem}\label{C2}
Let $ \phi(t) $ be a solution to \eqref{pMA}. If $ \hat g $ is a hyperk\"ahler metric compatible with $ (I,J,K) $ then
\[
\mathrm{tr}_{\hat g} g_\phi \leq C\,,
\]
for a uniform constant $ C>0 $ depending only on $ (M,I,J,K,g) $, $\hat g$ and $ f $.
\end{lem}
\begin{proof}
Consider the quantity
\[
Q:=\mathrm{tr}_{\hat g}g_\phi -A\phi\,,
\]
where $A$ is a constant to be chosen later. Assume $ M\times [0,T] $ is contained in the domain of $ \phi $ and let $ (x_0,t_0) $ be a maximum point of $ Q $ on $ M\times [0,T] $. We may assume $ t_0>0 $, otherwise the estimate is obvious. Fix normal coordinates at $ x_0 $ with respect to the hyperk\"ahler metric $\hat g$. Observe that the first derivatives of  $ J $ vanish at $x_0$. Now we compute at $x_0$
\begin{equation}\label{Delta}
\begin{aligned}
\Delta_\phi \mathrm{tr}_{\hat g}g_\phi&=g_\phi^{i\bar j} ((\hat g^{r\bar s})_{,i\bar j}g^\phi_{r\bar s}+\hat g^{r\bar s} g^\phi_{r\bar s, i \bar j})\\
&=-g_\phi^{i\bar j}\hat g^{a\bar s}\hat g^{r\bar b}\hat g_{a\bar b,i\bar j}g^\phi_{r\bar s}+g_\phi^{i\bar j}\hat g^{r\bar s}\left(g_{r\bar s,i\bar j}+\frac{1}{2}\phi_{r\bar s i\bar j} +\frac{1}{2}J_{r,i\bar j}^{\bar a}J_{\bar s}^b\phi_{b\bar a}+\frac{1}{2}J_{r}^{\bar a}J_{\bar s,i\bar j}^b\phi_{b\bar a}\right)\,.
\end{aligned}
\end{equation}
On the other hand
\begin{equation}\label{Delta2}
\begin{aligned}
\partial_t \mathrm{tr}_{\hat g}g_\phi=\, &\hat g^{r\bar s} \partial_t g^\phi_{r\bar s}=\hat g^{r\bar s}\dot \phi_{r\bar s}
=\hat g^{r\bar s} \left( g_\phi^{i\bar j}g^\phi_{i\bar j,r} \right)_{,\bar s}-2\Delta_{\hat g} f\\
=\, &-\hat g^{r\bar s}g_\phi^{i\bar l}g_\phi^{k\bar j}g^\phi_{k\bar l,\bar s}g^\phi_{i\bar j,r}+\hat g^{r\bar s}  g_\phi^{i\bar j}g^\phi_{i\bar j,r\bar s} -2\Delta_{\hat g} f\\
=\, &-\hat g^{r\bar s}g_\phi^{i\bar l}g_\phi^{k\bar j}g^\phi_{k\bar l,\bar s}g^\phi_{i\bar j,r}+\hat g^{r\bar s}  g_\phi^{i\bar j}\left(g_{i\bar j,r\bar s}+\frac{1}{2}\phi_{ i\bar jr\bar s} +\frac{1}{2}J_{i,r\bar s}^{\bar a}J_{\bar j}^b\phi_{b\bar a}+\frac{1}{2}J_{i}^{\bar a}J_{\bar j,r\bar s}^b\phi_{b\bar a}\right) \\
\,&-2\Delta_{\hat g} f\,.
\end{aligned}
\end{equation}

The metric $ \hat g $ is hyperk\"ahler, hence
the corresponding $ (2,0) $-form is closed, which in $ I $-holomorphic coordinates implies $ 0=-(\hat g_{a\bar c}J^{\bar c}_b)_{\bar k}=-\hat g_{a\bar c,\bar k}J^{\bar c}_b-\hat g_{a\bar c}J^{\bar c}_{b,\bar k} $ and derivating again we get $ \hat g_{a\bar c,\bar kl}J^{\bar c}_b+\hat g_{a\bar c,\bar k}J^{\bar c}_{b,l}+\hat g_{a\bar c,l}J^{\bar c}_{b,\bar k}+\hat g_{a\bar c}J^{\bar c}_{b,\bar kl}=0 $, which, at $ x_0 $, rewrites as
\[
\hat g^{a \bar q}\hat g_{a\bar p,l\bar k}=J^b_{\bar p}J^{\bar q}_{b,l\bar k}\,.
\]
This identity is simply expressing the fact that the curvatures of the Obata connection coincides with the one of the Levi-Civita connection, indeed in general the Christoffel symbols of the Obata connection are
\[
\Gamma^k_{ij}=-J^{\bar s}_{i,j}J^k_{\bar s}
\]
(see (2.27) in \cite{Dinew-Sroka}) and its curvature is
\[
R_{{\phantom i}\bar jk\bar l}^{\bar i}=-\partial_{k}\overline{\Gamma^i_{jl}}=\partial_{k}\left( J^{s}_{\bar j,\bar l}J^{\bar i}_{s} \right)=J^{s}_{\bar j,k\bar l}J^{\bar i}_{s}+J^{s}_{\bar j, \bar l}J^{\bar i}_{s,k}=J^{s}_{\bar j,k\bar l}J^{\bar i}_{s}=\left(J^{s}_{\bar j}J^{\bar i}_{s}\right)_{,k\bar l}-J^{s}_{\bar j}J^{\bar i}_{s,k\bar l}=-J^{s}_{\bar j}J^{\bar i}_{s,k\bar l}
\]
at a point where the first derivatives of $J$ vanish. 

Moreover taking into account that a hyperk\"ahler metric is Ricci-flat we have
\begin{equation}\label{Simp2}
\begin{aligned}
\hat g^{r\bar s}g_\phi^{i\bar j}\left(J_{i,r\bar s}^{\bar a}J_{\bar j}^b+J_{i}^{\bar a}J_{\bar j,r\bar s}^b\right)&=\hat g^{r\bar s}g_\phi^{p\bar q}J^{i}_{\bar q}J^{\bar j}_p\left(J_{i,r\bar s}^{\bar a}J_{\bar j}^b+J_{i}^{\bar a}J_{\bar j,r\bar s}^b\right)\\
&=-\hat g^{r\bar s}g_\phi^{b\bar q}J^{i}_{\bar q}J_{i,r\bar s}^{\bar a}-\hat g^{r\bar s}g_\phi^{p\bar a}J_p^{\bar j}J_{\bar j,r\bar s}^b\\
&=\hat g^{r\bar s}g_\phi^{b\bar q}R_{\phantom{a}\bar qr\bar s}^{\bar a}+\hat g^{r\bar s}g_\phi^{p\bar a}\overline{R_{\phantom{b}\bar ps\bar r }^{\bar b}}=0\,.
\end{aligned}
\end{equation}
Furthermore, we have
\begin{equation}\label{Simp3}
g^{i\bar j}_\phi R_{i\bar jr\bar s}=0\,,
\end{equation}
because
\[
g_\phi^{i\bar j}R_{i\bar jr\bar s}=-g_\phi^{i\bar j}\hat g_{i\bar k}J^{b}_{\bar j}J^{\bar k}_{b,r\bar s}=-g_\phi^{i\bar j}\hat g_{l\bar a}J^l_{\bar k}J^{\bar a}_iJ^{b}_{\bar j}J^{\bar k}_{b,r\bar s}=-g_\phi^{b\bar a}\hat g_{l\bar a}J^{l}_{\bar k}J^{\bar k}_{b,r\bar s}=-g_\phi^{b\bar a}\overline{R_{a\bar b s\bar r}}=-g_\phi^{b\bar a}R_{b\bar a r\bar s}\,,
\]
and we obtain
\begin{equation}\label{Simp4}
\begin{aligned}
g_\phi^{i\bar j}\hat g^{r\bar s}\left(J_{r,i\bar j}^{\bar a}J_{\bar s}^b+J_{r}^{\bar a}J_{\bar s,i\bar j}^b\right)&=g_\phi^{i\bar j}\hat g^{p\bar q}J_{\bar q}^rJ_p^{\bar s}\left(J_{r,i\bar j}^{\bar a}J_{\bar s}^b+J_{r}^{\bar a}J_{\bar s,i\bar j}^b\right)\\
&=-g_\phi^{i\bar j}\hat g^{b\bar q}J_{\bar q}^rJ_{r,i\bar j}^{\bar a}-g_\phi^{i\bar j}\hat g^{p\bar a}J_p^{\bar s}J_{\bar s,i\bar j}^b\\
&=g_\phi^{i\bar j}\hat g^{b\bar q}R_{\phantom{a}\bar qi\bar j}^{\bar a}+g_\phi^{i\bar j}\hat g^{p\bar a}\overline{R_{\phantom{b}\bar pj\bar i }^{\bar b}}=0\,.
\end{aligned}
\end{equation}

Therefore, \eqref{Delta} and \eqref{Delta2}, with the simplifications \eqref{Simp2}, \eqref{Simp3} and \eqref{Simp4} yield
\[
(\partial_t-\Delta_\phi)\mathrm{tr}_{\hat g}g_\phi=-\hat g^{r\bar s}g_\phi^{i\bar l}g_\phi^{k\bar j}g^\phi_{k\bar l,\bar s}g^\phi_{i\bar j,r}+\hat g^{r \bar s}g_\phi^{i\bar j}g_{i\bar j,r\bar s}-\hat g^{r\bar s}g_\varphi^{i \bar j} g_{r\bar s,i\bar j}
-2\Delta f \leq C + C \mathrm{tr}_{g_\varphi}\hat g\,,
\]
where $C>0$ do not depend on $\varphi$.
At the point $ (x_0,t_0) $ we then have
\[
0\leq \left(\partial_t-\Delta_\phi\right) Q=(\partial_t-\Delta_\phi) \mathrm{tr}_{\hat g}g_\phi-A \dot \phi+A\Delta_\phi \phi\leq C-\dot \phi+2nA+(C-A)\mathrm{tr}_{g_\phi}\hat g \,,
\]
and choosing $A > C$ we get
\[
\mathrm{tr}_{g_\phi}\hat g(x_0,t_0)\leq C
\]
because $ \dot \phi $ is uniformly bounded by Lemma \ref{C0}. This allows us to give a bound on $ \mathrm{tr}_{\hat g}g_\phi(x_0,t_0) $ by using 
\[
\mathrm{tr}_{\hat g}g_\phi(x_0,t_0)\leq \frac{1}{(2n-1)!}(\mathrm{tr}_{g_\phi}g(x_0,t_0))^{2n-1}\frac{\omega_\phi^{2n}(x_0,t_0)}{\hat \omega^{2n}(x_0,t_0)}\leq C\mathrm{e}^{\dot \phi +2f}(x_0,t_0)\leq C\,,
\]
see \cite[Corollary 3.3.5]{SW}, where we also used Lemma \ref{C0} again. Since $ Q(x,t)\leq Q(x_0,t_0) $ the claim follows. 
\end{proof}

\begin{proof}[Proof of Theorem $\ref{main}$]
Let $\varphi(t)$, $t\in [0,T)$, be the maximal solution to \eqref{pMA}.
Assume by contradiction that $T$ is finite.  In view of Lemmas \ref{C0}, \ref{C2} and Remark \ref{Schauder} $	\tilde \phi $ is uniformly bounded in $ C^k $ norm for every $ k $. Hence $ \tilde \phi $ is smooth at the time $ T $, but then short-time existence would imply that the solution exists on $ [0,T+\delta) $ for some $ \delta>0 $, contradicting the maximality of $ T $, hence $ T=\infty $.

The smooth convergence of $ \tilde \phi(t) $ to some $ \phi_\infty\in C^\infty(M) $ can be obtained repeating almost verbatim the argument of Gill \cite[Sections 6 and 7]{Gill}. The technique developed by Gill is inspired by Li and Yau \cite{LY} and is focused on studying the heat-type equation
\begin{equation}\label{heat}
\dot u = \Delta_\phi u\,.
\end{equation}
In \cite{Gill} a Harnack inequality is obtained and applied to $ u= \dot \phi $ in order to show that $ \dot{\tilde \phi} $ decays exponentially. This allows to deduce the convergence of $ \tilde\phi $ to a smooth function. We emphasise that for us the dependence of $ g_\phi $ from the potential $ \phi $ is not the same as in \cite{Gill}, however the argument never requires to express $ g_\phi $ in terms of the potential and the only thing that matters is that $ g_\phi $ is uniformly bounded in $ C^\infty $.

Therefore, since $ u=\dot \phi $ satisfies \eqref{heat} and we have $ C^\infty $ bounds  by Remark \ref{Schauder}, $ \tilde \phi $ converges smoothly to some function $ \phi_\infty\in C^\infty(M) $. Since $ \tilde \phi $ solves the equation
\[
\dot {\tilde
\phi}(t)=2\log \frac{(\Omega+ \partial \partial_J \tilde \phi(t))^n}{\Omega^n}-2f-\frac{2}{\int_M \Omega^n \wedge \bar \Omega^n}\int_M \left( \log \frac{(\Omega+ \partial \partial_J \tilde \phi(t))^n}{\Omega^n}-f \right) \Omega^n \wedge \bar \Omega^n
\]
taking the limit as $ t\to \infty $ yields
\[
(\Omega + \partial \partial_J \phi_\infty)^n= \mathrm{e}^{f+b} \Omega^n
\]
where
\[
b=\frac{1}{\int_M \Omega^n \wedge \bar \Omega^n}\int_M \left( \log \frac{(\Omega+ \partial \partial_J \phi_\infty)^n}{\Omega^n}-f \right) \Omega^n \wedge \bar \Omega^n\,. 
\]

In order to conclude the proof of Theorem \ref{main} we have still to observe that the equation has at most one solution.  Here we can work as follows: let $(\phi,b),(\psi,c)$ be two solutions to \eqref{QMA} with $ b\geq c $. We have that
\[
\frac{\partial \partial_J(\phi-\psi) \wedge \sum_{k=0}^{n-1} \Omega_\phi^k\wedge \Omega_\psi^{n-1-k}}{\Omega^n}= \frac{\Omega_\phi^n-\Omega_\psi^n}{\Omega^n}=(\mathrm{e}^b-\mathrm{e}^c)\mathrm{e}^{f} \geq 0\,.
\]
On the left hand-side we have a second order linear elliptic operator without free term applied to $ \phi-\psi $ and from the maximum principle and the fact that $ \sup_M \phi=\sup_M \psi=0 $ it follows $\phi=\psi $. Hence we have also $ b=c $ and uniqueness follows. 
\end{proof}

\section{The adapted Chern-Ricci flow}\label{sJ-}
In this section we consider flow \eqref{J-} on hypercomplex manifolds. 

\medskip 
Given a compact hyperhermitian manifold $(M,I,J,K,g)$ and a covariant $2$-tensor $S$, we denote by 
$$
S^{-}:=\frac12 (S-JS)
$$
its $J$-anti-invariant part in order to rewrite \eqref{J-} as  
$$
\dot \omega(t) =-{\rm Ric}^- (\omega(t))\,,\quad \omega(0)=\omega\,. 
$$
Analogously to the K\"ahler and the Hermitian case the flow is equivalent to a scalar one
\begin{equation}\label{equa}
\dot \varphi(t)=\log\frac{(\omega-t {\rm Ric}^-(\omega)+ (i\partial\bar\partial \varphi(t))^-)^{2n}}{\omega^{2n}}\,,\quad \varphi(0)=0\,. 
\end{equation}
Indeed, if $\varphi$ solves \eqref{equa}, then  $\omega(t):=\omega-t {\rm Ric}^-(\omega)+ (i\partial\bar\partial \varphi(t))^-$ solves \eqref{J-} since  
$$
\dot \omega(t)=-{\rm Ric}^- (\omega) + (i\partial\bar\partial \dot\varphi(t))^- = -{\rm Ric}^- (\omega) + \left(i\partial\bar\partial \log\frac{\omega(t)^{2n}}{\omega^{2n}}\right)^- = -{\rm Ric}^- (\omega(t))\,. 
$$
Conversely if $\omega(t)$ solves \eqref{J-}, then we define
$$
\varphi(t):=\int_0^t \log \frac{\omega(s)^{2n}}{\omega^{2n}}\, ds
$$
and we have 
$$
\begin{aligned}
\partial_t(\omega(t)  -\omega+t {\rm Ric}^-(\omega)- (i\partial\bar\partial \varphi(t))^-)&=-{\rm Ric}^- (\omega(t))+{\rm Ric}^-(\omega)-(i\partial\bar\partial \dot\varphi(t))^-\\
& =-{\rm Ric}^- (\omega(t))+{\rm Ric}^-(\omega)-\left(i\partial\bar\partial \log \frac{\omega(t)^{2n}}{\omega^{2n}}\right)^-=0\\
\end{aligned}
$$
which implies $\omega(t) = \omega-t {\rm Ric}^-(\omega) +  (i\partial\bar\partial \varphi(t))^-$ for all $t$.

\medskip
\noindent According to the K\"ahler \cite{TZ} and the Hermitian case \cite{TW},  it is quite natural to introduce the following conjecture 

\medskip 
\begin{conj}\label{conj}
There exists a unique maximal solution to the flow \eqref{J-} on 
$[0,T)$, where 
$$
T=\sup \,\{t\geq 0\,:\,\mbox{ there exists }\psi \in C^{\infty}(M) \mbox{ s.t. }\omega-t{\rm Ric}(\omega)^{-}+i(\partial\bar\partial\psi)^{-}>0 \}\,.
$$
\end{conj}

Note that, if $c_1^{\rm BC}(M,I)=0$, then flow \eqref{J-} is equivalent to  the parabolic quaternionic Monge-Amp\`ere equation \eqref{pMA}. In particular Theorem \ref{main}  implies that if $(M,I,J,K)$ has an underlying hyperk\"ahler metric, then \eqref{J-} has a long-time solution and the conjecture in this special case is verified. In order to prove the conjecture in the general case we need a priori estimates on the solution $\varphi$ to \eqref{equa}. We can observe that  estimates can be obtained working as in the complex case \cite[Lemma 4.1]{TW} except for the estimate for the second order derivatives which we leave open.
%
\begin{prop}\label{time}
Let $ (M,I,J,K,g) $ be a compact hyperhermitian manifold and $ \omega(t) $ a solution of \eqref{J-}. Assume that there exists a uniform positive constant $ C $ such that
\[
C^{-1}\omega \leq \omega(t) \leq C \omega\,,
\]
then Conjecture \ref{conj} holds.
\end{prop}

Flow \eqref{equa} fits in the following quite general class of parabolic problems: \\ 
let $(M,g)$ be a compact Riemannian manifold and 
$$
F_{t}\colon C^{2}_{+}(M)\to C^{0}(M)\,,\quad t\in [0,T)\,,
$$
a smooth family of  second-order partial differential operators defined on an open subset $C^{2}_{+}(M)$ of $C^{2}(M)$. Assume that $-F_{t}\colon C^{2}_{+}(M)\to C^{0}(M)$ is strongly elliptic for every $t\in [0,T)$. Assume further that 
$$
F_t(\psi+C)= F_t(\psi)
$$ 
for every $\psi\in C^{2}_+(M)$, $t\in [0,T)$ and constant $C$. Then we consider the parabolic flow  
\begin{equation}\label{general}
\dot \varphi(t)=F_t(\varphi(t))\,,\quad \varphi(0)=0\,.
\end{equation}

\begin{lem}\label{genC0}
Assume that there exists a continuous map $\Lambda\colon M\times [0,T)\to \R$ such that for any $0<T'< T$ and $\psi\in C^{2}(M\times [0,T'])$ 
$$
F_{t_0}(\psi({t_0}))(x_0)\leq \Lambda(x_0,t_0) \mbox{ if $(x_0,t_0)$ is a maximum point of } \psi\,,
$$
then solutions to \eqref{general} satisfy a uniform upper bound. Analogously, if there exists a continuous map $\lambda\colon M\times [0,T)\to \R$ such that for any $0<T'< T$ and $\psi\in C^{2}(M\times [0,T'])$
$$
F_{t_0}(\psi({t_0}))(x_0)\geq \lambda(x_0,t_0) \mbox{ if $(x_0,t_0)$ is a minumum point of } \psi\,,
$$
then solutions to \eqref{general} satisfy a uniform lower bound.
\end{lem}
\begin{proof}
Let $\phi\in C^2(M\times [0,T'))$ be a solution of \eqref{general} with $T'<T$. Fix $0<T''<T'$ and consider $\psi(t):=\varphi(t)-At$, where $A>{\rm max}_{M\times [0,T']} \Lambda$ is a positive constant. Let $(x_0,t_0)$ be a maximum point of $\psi$ in $M\times [0,T'']$. Since 
$$
\partial_t\psi(x_0,t_0)=F_{t_0}(\varphi(t_0))(x_0)-A=F_{t_0}(\psi(t_0))(x_0)-A\leq \Lambda(x_0,t_0)-A<0\,,
$$
then $t_0=0$ and, since $\psi_0\equiv 0$, 
$$
\varphi(x,t)\leq t A\leq T'A 
$$
for every $(x,t)\in M\times[0,T'']$. Since $T''$ is arbitrary the upper bound on $\varphi$ follows.  

In a similar way, considering $\eta(t)=\phi(t)+Bt$, where $ B>-\min_{M\times [0,T']}\lambda $, at a minimum point of $\eta$ in $ M\times [0,T''] $ we achieve a lower bound.
\end{proof}


\begin{lem}\label{genC0t}
Let $\varphi\in C^{2}(M\times [0,T'))$ be a solution to \eqref{general} uniformly bounded from above, where $0<T'<T$. Assume that there exists a continuous function $g\colon [0,T]\to \R$ such that for every $ T'<T''<T $ we have
$$
(T''-t)(\partial_tF_t)(\varphi(t))+g(t)-(F_{t})_{*|\varphi(t)}\varphi(t)>0 \mbox{ for all }t\in [0,T')\,,
$$
then the first time derivative of $\varphi$ satisfies a uniform lower bound. Analogously, if $\varphi\in C^{2}(M\times [0,T'))$ is a solution to \eqref{general} uniformly bounded from below and there exists a continuous function $ h\colon [0,T] \to \R $ such that
\[
t(\partial_tF_{t})(\varphi(t))-h(t)+(F_t)_{*\vert \phi(t)} \phi(t)\leq 0 \mbox{ for all }t\in [0,T')\,,
\]
then the first time derivative of $\varphi$ satisfies a uniform lower bound. 
\end{lem}
\begin{proof}
Let $G$ be a primitive function of $g$ such that $G(0)=0$. Fix $ T'<T''<T $ and define 
$$
Q(t)=(T''-t)\dot \varphi(t) +\varphi(t)+ G(t)
$$
Then 
$$
\partial_tQ(t)=(T''-t)\ddot{\varphi}(t)+g(t)
$$
and
$$
(F_{t})_{*|\varphi(t)}Q(t)=(T''-t)(F_{t})_{*|\varphi(t)}\dot{\varphi}(t)+(F_{t})_{*|\varphi(t)}\dot\varphi(t)
$$
Since $\dot \varphi(t)=F_t(\varphi(t))$, we have $\ddot \varphi(t)=(F_{t})_{*|\varphi(t)}\dot{\varphi}(t)+(\partial_tF_{t})(\varphi(t))$.
Therefore using our assumptions
$$
\partial_tQ(t)-(F_t)_{*|\varphi(t)}Q(t)=(T''-t)(\partial_tF_{t})(\varphi(t))+g(t)-(F_{t})_{|*\varphi(t)}\dot\varphi(t)>0
$$
and by the maximum principle 
$$
(T''-t)\dot \varphi(t)+\varphi(t)+G(t)=Q(t)\geq \inf_{M} Q(0)=T''\inf_{M}\,\dot \varphi(0)\,,
$$
for any $t\in [0,T')$, i.e.,  
$$
(T''-t)\dot \varphi(t)\geq T''\inf_{M}\,\dot \varphi(0)-\varphi(t)-G(t)\,. 
$$
Since $\varphi$ is uniformly bounded from above, then we have 
$$
(T''-t)\dot \varphi(t)\geq -C\ 
$$
for a uniform positive constant $C$. Hence 
$$
\dot \varphi(t)\geq -\frac{C}{T''-T'}
$$
and the claim follows.   

For the lower bound of $ \dot \phi $ we consider the quantity
\[
S(t)=t\dot \phi(t)-\phi(t)-H(t)
\]
where $ H $ is a primitive of $ h $ such that $ H(0)=0 $ and proceed analogously.
\end{proof}

%

\begin{proof}[Proof of Proposition \ref{time}]
Uniqueness of solutions to \eqref{J-} follows from the same property of the equivalent flow \eqref{equa}. To prove that the solution $\omega(t)$ to \eqref{J-} exists on $[0,T)$ it is enough to focus on the parabolic Monge-Ampère-type equation \eqref{equa}. By standard parabolic theory there is a unique solution $\phi(t)$ to \eqref{equa} on a maximal time interval $ [0,T') $. Assume by contradiction $T'<T$. Taking
\[
\Lambda=\lambda=\log \frac{ (\omega- t \Ric^-(\omega))^{2n}}{\omega^{2n}}
\]
Lemma \ref{genC0} applies to flow \eqref{equa}. Furthermore, taking $ g(t)=h(t)=2n $ Lemma \ref{genC0t} also applies. Combining these with the assumption $C^{-1}\omega \leq \omega(t)\leq C\omega$ and working in the same fashion as in Theorem \ref{C2alpha} and applying \cite[Theorem 5.1]{Chu} we obtain uniform estimates on $\phi(t)$ of any order on $ [0,T') $. In particular $ \lim_{t\to T'} \phi(x,t)$ is still smooth and short time-existence gives a contradiction, since we would then be able to extend the flow on $ [0,T'+\epsilon) $ for some $\epsilon>0$. 
\end{proof}


\section{Remarks and Further Developments}

From the geometric point of view the conjecture of Alesker and Verbitsky consists in prescribing the $J$-anti-invariant part of the Chern-Ricci tensor of an HKT metric in a fixed quaternionic Bott-Chern cohomology class. The existence of Chern-Ricci flat HKT metrics on a hypercomplex manifold implies that the canonical bundle is holomorphically trivial. When the canonical bundle is not holomorphically trivial, it is quite natural to study the existence of hyperhermitian metrics $\omega$ satisfying the Einsten-type condition 
\begin{equation}\label{einstein}
{\rm Ric}-J{\rm Ric}=\lambda\omega\,,\mbox{ for some constant }\lambda\,,  
\end{equation}
or, more generally, of hyperhermitian metrics with constant Chern-scalar curvature (this research project has been suggested to the second named author by Misha Verbitsky). Equation \eqref{einstein}
can be rewritten in terms of $\Omega$ as 
\begin{equation}\label{Einstein}
\partial_J\theta^{1,0}=\lambda \Omega\,,
\end{equation} 
where $\theta$ is the Lee form.   
In the compact case \eqref{Einstein} forces $\lambda$ to be non-negative and on Joyce homogeneous examples \cite{Joyce} (which are the simplest examples of compact HKT manifolds where the canonical bundle is not holomorphically trivial) the canonical metric satisfies  \eqref{Einstein} \cite{G}. Since $\lambda$ is non-negative, in analogy with the K\"ahler case, we expect that it is possible to find obstructions to the existence of HKT metrics satisfying \eqref{Einstein}. 

\medskip 
In order to study the existence of HKT metrics having constant Chern scalar curvature it is quite natural to consider the following analogue of the Calabi-flow on HKT manifolds 
\begin{equation}\label{CF}
\dot \phi(t)=s_\phi(t)\,,\quad \varphi(0)=0\,,
\end{equation}
where $s_\phi(t)$ is the Chern-scalar curvature of $\Omega_\phi:=\Omega + \partial\partial_J\varphi(t)$ and $\Omega$ is a fixed HKT metric. 
When the canonical bundle is holomorphically trivial \eqref{CF} is the gradient flow of the following Mabuchi-type functional 
$$
\mathcal M(\varphi)= \int_M\log \frac{\Omega_\varphi^n}{\Omega^n}\,\Omega_{\varphi}^n\wedge \bar \Theta-\int_{M} h\,\Omega_{\varphi}^n\wedge\bar\Theta\,,
$$
where $h$ is a $\partial\partial_J$-potential of $\partial_J\theta^{1,0}$ and $\Theta$ is a holomorphic volume form.  

\end{document}